\documentclass{amsart}

\usepackage[english]{babel}
\usepackage{amsfonts,amssymb,amsmath}
\usepackage{mathrsfs} 

\newcommand{\erre} {{\mathbb {R}}}
\newcommand {\elle} {{\mathscr {L}}}
\newcommand {\elleo} {{\mathscr{L}_0}}
\def \acca{{\mathcal {H}}}
\def \cappa{{\mathcal {K}}}
\def\erren{{\erre^{ {N} }}}
\def\erreu{{\erre^{ {N+1} }}}
\def\av{``}
\def\cv{''}
\def\inn{\mbox{ in }}
\def\ass{\mbox{ as }}

\def\andd{ \quad\mbox{ and } \quad }
\def\ifff{ \quad\mbox{ if } \ }

\newcommand{\tende}{\rightarrow}
\newcommand{\ttende}{\longrightarrow}
\newcommand{\enne} {\mathbb{N}}
\newcommand{\frecciaf} {\longmapsto}
\newcommand{\inte} {\cap}
\newcommand\de{\partial}

\newcommand\partialj{\partial_{x_j}}
\newcommand\partialt{\partial_{t}}
\newcommand{\meno} {\smallsetminus}

\newcommand\eps{\varepsilon}
\newcommand\sigmastar{\sigma^\ast\!\!}

\newtheorem{theorem}{Theorem}[section]
\newtheorem{proposition}[theorem]{Proposition}

\newtheorem{lemma}[theorem]{Lemma}
\theoremstyle{remark}

\theoremstyle{definition}

\numberwithin{equation}{section}

\title[ On the Dirichlet problem in cylindrical domains
for evolution  PDE's]{ On the Dirichlet problem in cylindrical domains
for evolution Ole\v{\i}nik--Radkevi\v{c}  PDE's:\\ a Tikhonov-type theorem}

\author{Alessia E. Kogoj}
\address{Dipartimento di Scienze Pure e Applicate (DiSPeA)\\ 
				 Universit\`{a} degli Studi di Urbino ``Carlo Bo''\\
				 Piazza della Repubblica, 13 - 61029 Urbino (PU), Italy.}
\email{alessia.kogoj@uniurb.it}

\subjclass[2010]{35H10; 35K65; 35J70; 35J25; 31D05; 35D99}
\keywords{Dirichlet problem, Perron-Wiener solution, Boundary behavior of Perron-Wiener solutions, Hypoelliptic operators, Potential theory}

\begin{document}

\begin{abstract}

We consider the linear second order PDO's 
$$ \mathscr{L}  =  \mathscr{L}_0 - \partial_t : = \sum_{i,j =1}^N  \partial_{x_i}(a_{i,j}  \partial_{x_j} ) - \sum_{j=i}^N b_j  \partial_{x_j} - \partial _t,$$and assume  that $\mathscr{L}_0$ has nonnegative characteristic form and satisfies the Ole\v{\i}nik--Radkevi\v{c} rank hypoellipticity condition. These hypotheses allow the construction of Perron-Wiener solutions of the Dirichlet problems for $\mathscr{L}$ and $\mathscr{L}_0$ on bounded open subsets of $\mathbb R^{N+1}$  and of $\mathbb R^{N}$, respectively. 

 Our main result is the following Tikhonov-type theorem:  \\Let $\mathcal{O}:= \Omega \times ]0, T[$ be a bounded cylindrical domain of $\mathbb R^{N+1}$,
$\Omega \subset \mathbb R^{N},$   $x_0 \in \partial \Omega$ and \mbox{$0 < t_0 < T$}. Then $z_0 = (x_0, t_0) \in \partial \mathcal{O}$ is $\mathscr{L}$-regular for $\mathcal{O}$ if and only if $x_0$ is $\mathscr{L}_0$-regular for $\Omega$.

As an application, we derive a boundary regularity criterion for  degenerate Ornstein--Uhlenbeck operators.

\end{abstract}

\maketitle

%% keywords here, in the form: keyword \sep keyword

%% PACS codes here, in the form: \PACS code \sep code

%% MSC codes here, in the form: \MSC code \sep code
%% or \MSC[2008] code \sep code (2000 is the default)

%% \linenumbers

%% main text

\section{Introduction}
We  consider linear second order partial differential operators  of the type
 
 \begin{equation}\label{staz}  \mathscr{L}_0 :  =\sum_{i,j=1}^N  \partial_{x_i} \left(a_{ij}\partial_{x_j} \right)  + \sum_{j=1}^N b_j \partial_{x_j} 
\end{equation}
in an open set  $X$ of $\erre^N$, $N\geq 2,$ and their \av evolution\cv counterpart in $X\times \erre$
 \begin{equation}\label{evol}  \mathscr{L}  =   \mathscr{L}_0 - \de_t.
\end{equation}

We assume $ \mathscr{L}_0$ in \eqref{staz} is of non totally degenerate Ole\v{\i}nik and Radkevi\v{c} type, i.e., we assume 

\begin{itemize}
\item[(H1)]  $a_{ij}=a_{ji}, b_i \in C^\infty(X,\erre)$ and 

$$A(x):=(a_{ij}(x))_{i,j=1,\ldots,N}\geq 0\qquad \forall x\in X.$$

Moreover $$\inf_X a_{11}=: \alpha >0.$$
\item[(H2)] $  \mathrm{rank\ } \mathrm{Lie}  \{ X_1,\ldots,X_N,X_0 \}(x)=N\qquad \forall x\in X,$ 
where, $$X_i=\sum_{j=1}^N a_{ij}\partialj,\  i=1,\ldots,N,\andd X_0=\sum_{j=1}^Nb_j\partialj.$$

\end{itemize}
Hypotheses (H1) and (H2) imply that $\elleo$ is  hypoelliptic in $X$ (see \cite{OR}), that is:

\begin{center}$\Omega$ open subset of $X$, $u\in \mathcal{D}'(\Omega), \elleo u \in C^\infty(\Omega, \erre) \implies u\in C^\infty(\Omega, \erre).$ \end{center}
The same assumptions (H1) and (H2) also imply that $\elleo - \partialt$ is hypoelliptic in $X\times \erre$.  

We will show in Section 2 that $\elleo$ and $\elleo - \partialt$ endow $X$ and $X\times \erre$, respectively, with a local structure of $\sigma^\ast\!\!$-harmonic space, in the sense of [3], Chapter 6.  As a consequence, in particular, the Dirichlet problems 

\begin{equation*} 
\begin{cases}
 \elleo   u= 0 \mbox{ in } \Omega,    \\
  u|_{\partial \Omega} = \varphi, \end{cases}\andd
  \begin{cases}
 (\elleo -\partialt)  v= 0 \mbox{ in } \mathcal{O}:=\Omega  \times ]0,T[, \\
  v|_{\partial \mathcal{O}} = \psi, \end{cases}
\end{equation*}
have a generalized solution in the sense of Perron--Wiener, for every bounded open set $\Omega \subset \subset X,$ for every $T>0$, and for every $\varphi\in C(\partial \Omega, \erre)$ and \mbox{$\psi\in C(\partial \mathcal{O}, \erre)$}. We will denote such generalized solutions by, respectively, 
$$H^\Omega_\varphi \andd K^{\mathcal{O}}_\psi.$$
As usual, we say that a point $x_0\in \partial\Omega$ ($(x_0,t_0)\in \partial {\mathcal{O}}$) is $\elleo$-regular for $\Omega$ ($\elle$-regular for ${\mathcal{O}}$) if 
$$\lim_{x\ttende x_0} H^\Omega_\varphi(x)=\varphi (x_0)\qquad \forall \varphi\in C(\partial\Omega,\erre)$$ 
$$\left(\lim_{(x,t)\ttende (x_0,t_0)} K^{\mathcal{O}}_\psi(x,t)=\psi (x_0,t_0)\qquad \forall \psi\in C(\partial {\mathcal{O}},\erre)\right).$$

The aim of this paper is to prove the following theorem:
\begin{theorem}\label{main} Let $\Omega$ be a bounded open set with $\overline \Omega\subseteq X$, and let $x_0\in \partial\Omega$ and $t_0\in ]0,T[$. Then,  $x_0$ is $\elleo$-regular for $\Omega$ if and only if  $(x_0,t_0)$ is \mbox{$\elleo-\partialt$-regular} for ${\mathcal{O}}:=\Omega\times ]0,T[$.  \end{theorem}

When $\elle=\varDelta - \partial_t$ is the classical heat operator, our result re-establishes 
a theorem proved by Tikhonov in 1938 \cite{tychonov}.  Other proofs of the Tikhonov Theorem  were given by Fulks in 1956 and in 1957 \cite{fulks_56,fulks} and by Babu\v{s}ka and V\'{y}born\'{y} in 1962 \cite{babuska}.  Chan and Young extended the  Tikhonov Theorem to parabolic operators with H\"older continuous coefficients in 1977 \cite{chan_young}, and Arendt to parabolic operators with bounded measurable coefficients in 2000 \cite{arendt}. The corresponding version for $p$-Laplacian-type evolution operators has been proved by Kilpel\"{a}inen and Lindqvist  in 1996  \cite{kilp_lind}
 and by Banerjee and Garofalo in 2015 \cite{ban_gar}.

To the best of our knowledge, the only Tikhonov-type theorem for second order \av evolution\cv   sub-Riemannian PDO's appearing  in the literature is the result by Negrini \cite{negrini} in abstract $\beta$-harmonic spaces\footnote{For a definition of $\beta$-harmonic spaces see \cite{CC}.}.

This paper is organised  as follows. In Section 2, all the notions and results from
Potential Theory that we need are briefly recalled.  In particular, we recall the notion of 
$\sigmastar$-harmonic space  and then we prove that $\elle_0$ and $\elle$ endow $X$ and $X\times \erre$, 
respectively, with a local structure of $\sigmastar$-harmonic space. In this way, we derive the existence of a generalized solution in the sense of Perron--Wiener in both our settings.   Section 3 is devoted to two key results for the proof of the main theorem (Theorem \ref{main}), which is the content of  Section 4. Finally, combining our Tikhonov-type  theorem with a corollary of the Wiener--Landis-type criterion for Kolmogorov-type operators proved in \cite{kogoj_lanconelli_tralli_2017}, we establish a geometric boundary regularity criterion for degenerate Ornstein--Uhlenbeck operators.

\section{$\elleo$-harmonic and $\elle$-harmonic spaces}\label{sigmastar} 
\subsection{The $\sigma^\ast\!\!$-harmonic space}\label{prima}For the readers' convenience we recall the definition of $\sigma^\ast\!\!$-harmonic space supported on a an open set $E\subseteq\erre^p, p\geq 2$,  and refer  to Chapter 6 of the monograph \cite{BLU} for details.

 Let $\acca$ be a sheaf of functions in $E$ such that $\acca(V)$ is a linear subspace of $C(V,\erre)$, for every open set $V\subseteq E$. The functions in $\acca(V)$ are called {\it $\acca$-harmonic in $V.$} The open set $V$ is called  {\it $\acca$-regular} if 
\begin{itemize}
\item[$(i)$] $\overline V\subseteq E$ is compact;
\item[$(ii)$] for every $\varphi\in C(\partial V, \erre)$  there exists a unique function such that  \begin{center} \mbox{$h_\varphi^V(x)\tende \varphi(\xi)$} as $x\tende \xi$, for every $\xi\in\partial V;$\end{center}
\item[$(iii)$] $h_\varphi^V\geq 0$ if $\varphi\geq 0.$

\end{itemize}

A lower semicontinuous function $u: W\ttende ]-\infty, \infty],$ $W\subseteq E$ open, is called  {\it $\acca$-superharmonic} if 
\begin{itemize}
\item[$(i)$] $u\geq h_\varphi^V$ in $V$ for every $\acca$-regular open set $V$ with $\overline V\subseteq W$ and for every $\varphi\in C(\partial V,\erre)$ with $\varphi\leq u|_{\partial V};$
\item[$(ii)$] $\{x\in W\ |\ u(x)<\infty\}$ is dense in $W.$
\end{itemize}

We denote by $\overline{\acca}(W)$ the cone of the $\acca$-superharmonic functions in $W.$ 

The couple $(E,\acca)$ is called a $\sigmastar$-harmonic space if the following axioms hold:
\begin{itemize}
\item[(A1)] There exists a function $h\in \acca(E)$ such that $\inf h>0$.

\item[(A2)] If $(u_n)_{n\in\enne}$ is a monotone increasing sequence of $\acca$-harmonic functions in an open set $V\subseteq E$ such that $$\{x\in V\ |\ \sup_{n\in\enne} u_n(x)<\infty\}$$ is dense in $\Omega$, then 
$$u:=\sup_V u_n \mbox{\quad is $\acca$-harmonic in\ }  V.$$
\item[(A3)]  The family of the $\acca$-regular open sets is a basis of the Euclidean \mbox{topology} on $E$.
\item[(A4)] For every $x,y \in E$, $x\neq y$, there exist two nonnegative $\acca$-superharmonic and continuous functions $u,v$ in $E$ such that 
$$u(x)v(y)\neq u(y)v(x).$$
\item[(A5)] For every $x_0\in E$ there exists a nonnegative $\acca$-superharmonic and continuous function $S_{x_0}$ in $E$, such that 
$S_{x_0}(x_0)=0$ and $$\inf_{E\meno V} S_{x_0}>0$$
for every neighborhood $V$ of $x_0$. 

\end{itemize}
We now recall some crucial  results in $\sigmastar$-harmonic space theory; first of all the definition of Perron--Wiener solution to the Dirichlet problem.

Let $V$ be a bounded open set with $\overline V\subseteq E$, and let $\varphi: \partial V\ttende \erre$ be a bounded lower semicontinuous or upper semicontinuous function. Define 
$${{\mathcal{\overline U}}}^V_\varphi = \{ u\in \overline\acca(V)\ |\ \liminf_{x\ttende \xi} u(x) \geq \varphi (\xi) \quad \forall \xi\in\partial V\}$$ 
and 
\begin{equation} \label{pwinf} H_\varphi^V=: \inf \mathcal{\overline U}^V_\varphi.
\end{equation} 
Then $H_\varphi^V$ is $\acca$-harmonic in $\Omega.$ It is called the generalized Perron--Wiener solution to the Dirichlet problem  
\begin{equation*} 
\begin{cases}
 u\in \acca(V),   \\
  u|_{\partial V} = \varphi. \end{cases}\end{equation*}We also have 

\begin{equation} \label{pwsup} H_\varphi^V=: \sup \mathcal{\underline U}^V_\varphi,
\end{equation} 
where,
$${{\mathcal{\underline U}}}^V_\varphi = \{ v \in \underline\acca(V)\ |\ \limsup_{x\ttende \xi} v(x) \leq \varphi (\xi) \quad \forall \xi\in\partial V\}.$$ 
Here $\underline\acca(V):=-\overline\acca(V)$ denotes the cone of the $\acca$-{\it subharmonic} functions in $V.$

A point $y\in\partial V$ is called {\it $\acca$-regular}  for $V$ if 
$$\lim_{x\ttende y} H_\varphi^V(x)=\varphi(y)\qquad \forall \varphi\in C(\partial V,\erre).$$ 

On the $\sigmastar$-harmonic space Bouligand Theorem holds. Indeed: {\it a point $y\in\partial V$ is $\acca$-regular for $V$ if and only if there exists a $\acca$-barrier for $V$ at $y$}, i.e., if there exists a function $b$ $\acca$-superharmonic in $V\inte W,$ where $W$ is a neighborhood of $y,$ such that  

\begin{itemize}
\item[$(i)$] $b$ is $\acca$-superharmonic;
\item[$(ii)$] $b(x)>0\  \forall x\in V\inte W$ and $b(x)\ttende 0$ as $x\ttende y.$ 

\end{itemize}

For our purposes it is important to recall that if $y\in\partial V$ is $\acca$-regular for $V$ there exists a barrier function for $V$ at $y$ which is defined and $\acca$-harmonic all over $V.$

Finally, we recall the {\it minimum principle} for  $\acca$-superharmonic functions. 

Let $V$ be a bounded open set with $\overline V\subseteq E$ and let $u\in \overline \acca(V).$ If 
$$\liminf_{x\ttende y} u(x) \geq 0\quad \forall y\in \partial V,$$ then $u\geq 0$ in $V.$

\subsection{The $\elleo$-harmonic space}\label{elleo} 

Let $E$ be a bounded open subset of $X$ such that $\overline E\subseteq X.$ For every open set $V\subseteq E$ we let 

$$\acca(V)=\{ u\in C^\infty (V,\erre)\ |\ \elleo u = 0\inn V\}.$$
Then, $V\frecciaf \acca(V)$ is a a sheaf of functions such that $\acca(V)$ is a linear subspace of $C(V,\erre).$

If $u\in \acca(V)$ we will say that $u$ is $\acca$-harmonic or $\elle_0$-harmonic in $V.$

We have that 
\begin{equation}\label{sigmastar} (E,\acca) \mbox{  is a {\it $\sigmastar$-harmonic space}}.\end{equation} 
Before showing this statement we remark that a $C^2$-function $u$ in a open set $V$ is $\acca$-superharmonic if and only if 
$\elleo u \leq 0$ in $V$. This is a easy consequence of Picone's maximum principle (see e.g. \cite{kogoj_polidoro}, page 547). Now we are ready to prove \eqref{sigmastar}.

(A1) is satisfied since the constant functions are $\elleo$-harmonic.

(A2) -(A4) are proved in  \cite{kogoj_polidoro}. We would like to stress that our operators $\elleo$ are contained in the class considered in 
\cite{kogoj_polidoro} since the rank condition (H2) implies that both  $\elleo$ and  $\elleo - \beta,$ for every $\beta\geq 0,$ are hypoelliptic.

The axiom (A5) follows from the following Lemma which seems to have an independent interest in its own right.

\begin{lemma} \label{due!} Let us consider a linear second order PDO of the kind 
  \begin{equation*}  \mathcal{L} :  =\sum_{i,j=1}^N  a_{ij} \partial_{x_i x_j}   + \sum_{j=1}^N b_j \partial_{x_j}, 
\end{equation*}
where $a_{ij}=a_{ji}, b_j$ are continuous functions in $\overline Y$, where $Y$ is a bounded open subset of $\erren$. Suppose 
$$\inf_Y a_{11}:=\alpha >0 \andd \sum_{j=1}^N a_{jj} >0 \inn Y\footnote{We don't require $(a_{ij})_{i,j=1,\ldots,N}$ to be nonnegative definite. }.$$
Then, for every $x_0 \in Y$ there exists a function $h\in C^\infty (Y,\erre)$ such that 

\begin{itemize}
\item[$(i)$] $h(x_0)=0$ and $h(x)>0$ for every $x\neq x_0;$ 
\item[$(ii)$] $\mathcal{L} h >0$ in $X.$ 

\end{itemize}
\end{lemma} 
\begin{proof}For the sake of simplicity we assume $x_0=0.$ We define 
$$h(x)=E(\lambda x_1) + (x_2^2+\cdots + x_N^2), \quad x=(x_1,x_2, \ldots, x_N)\in \erren,$$
where $\lambda>0$ will be fixed below. Moreover, 
$$E(s)=\exp(\phi(s))- \exp(\phi(0))$$ and 
$$\phi(s)=\sqrt{1+s^2},\quad s\in \erre.$$ 
We have:
$$\phi(0)=1,\quad \phi(s)>1\quad \forall s\neq 0,\quad E(s)>0\quad \forall s\neq 0,\quad E(0)=0,$$
$$\phi'(s)=\frac{s}{\sqrt{1+s^2}},\quad \phi''(s)=\frac{1}{({1+s^2})^{\frac{3}{2}}}.$$
Hence 
$${\phi'}^2 +  \phi''= \frac{s^2}{{1+s^2}} + \frac{1}{({1+s^2})^{\frac{3}{2}}}\geq \frac{1}{2\sqrt{2}}\quad \forall s\in\erre.$$
On the other hand 
$$ E'=\exp (\phi)\phi',\quad E''=\exp (\phi)({\phi'}^2 +  \phi'').$$
Therefore, letting 
$$\beta:=\sup_X \sum_{j=1}^N |b_j|\qquad (<\infty) \andd \lambda =\sup_{x\in\overline{X}}|x|,$$ 
we get 
\begin{eqnarray*} 
\mathcal{L} h(x)&=& \lambda^2 E''(\lambda x_1) a_{11}(x) + \lambda E'(\lambda x_1)b_1 + 2 \sum_{j=2}^N(a_{jj}(x)+b_j(x)x_j)\\
\\&\geq& \exp(\phi(\lambda x_1)) \left( \frac{a_{11}(x)}{2\sqrt{2}} \lambda^2 -\lambda |b_1| \right) -2\sum_{j=2}^N |b_j||x_j| \\
&\geq&  \lambda^2 \left(\frac{\alpha}{2\sqrt{2}} - \frac{|b_1|}{\lambda}\right) - 2 \beta \lambda\\
&\geq&  \lambda^2 \left(\frac{\alpha}{2\sqrt{2}} - \frac{\beta}{\lambda}\right) - 2 \beta \lambda.
\end{eqnarray*} 
If $\lambda$ is big enough, this implies 
$$\mathcal{L}h>0\inn X.$$ 
Moreover 
$$h(0)=E(0)=0,\quad h(x)>0 \ifff x>0.$$
The proof is complete. \end{proof}

\subsection{The $\elle$-harmonic space} 

Let $\widehat E$ be a bounded open subset of $X\times\erre$ such that $\overline{\widehat E}\subseteq X\times\erre.$ For every open set 
$V\subseteq \widehat E$ we let

$$\cappa(V)=\{ u\in C^\infty (V,\erre)\ |\ \elle u = 0 \inn V\}.$$
Then, $V\frecciaf \cappa(V)$ is a a sheaf of functions making 

\begin{equation*} (\widehat{E},\cappa) \mbox{ a {\it $\sigmastar$-harmonic space}}.\end{equation*} 

This can be proved just by proceeding as in subsection \ref{elleo}. We call $\cappa$-harmonic or $\elle$-harmonic in a open set $V$ the solutions to $\elle u =0$ in $V.$

Here we prove some typical results of the present $\cappa$-harmonic space, that we will need in the proof of the main theorem of this paper.  We first show a \av parabolic\cv minimum principle for $\elle$-subharmonic functions in cylindrical domains.

\begin{proposition}\label{pallino} Let $\Omega$ be a bounded open subset of $X$ such that $\overline\Omega\subseteq X$  and let $T>0$. Consider the cylindrical domain ${\mathcal{O}}:=\Omega\times ]0,T[$ and define the  \av parabolic    boundary\cv  of ${\mathcal{O}}$ as follows 
$$\partial_p {\mathcal{O}} : = (\Omega\times \{0\})\times (\partial \Omega\times ]0,T]).$$ Then, if $u\in\overline{\cappa}({\mathcal{O}})$ is such that 
$$\liminf_{z\ttende\zeta} u(z)\geq 0\quad\forall \zeta\in\partial_p {\mathcal{O}},$$
we have $u\geq 0$ in ${\mathcal{O}}.$ 
\end{proposition}
\begin{proof}For every arbitrarily fixed $\widehat{T}\in ]0,T[$ we let $\widehat{{\mathcal{O}}}=\Omega\times ]0,\widehat{T}[.$ We will prove that $u\geq 0$ in $\widehat{{\mathcal{O}}}$. Since $\widehat{T}$ is  arbitrarily fixed in $]0,T[$, this will give the proof of our lemma.  To this end, given any $\eps>0$, we define 
$$u_\eps (z)= u_\eps (x,t) :=  u(x,t) + \frac{\eps}{\widehat{T}-t},\quad z\in\widehat{{\mathcal{O}}}.$$
Since $u$ is $\cappa$-superharmonic in ${\mathcal{O}}$ and 
$$\elle  \frac{\eps}{\widehat{T}-t} = - \eps \de_t  \frac{1}{\widehat{T}-t}= -   \frac{\eps}{(\widehat{T}-t)^2}<0 \inn \widehat{{\mathcal{O}}}, $$
then $u_\eps$ is $\cappa$-superharmonic in ${\mathcal{O}}$. Moreover 
$$\liminf_{z\ttende\zeta} u_\eps (z)\geq 0\quad\forall \zeta\in\partial_p \widehat{{\mathcal{O}}},$$
and, for every $\xi \in \Omega,$ 
$$\liminf_{z\ttende (\xi,\widehat{T})} u_\eps (z)\geq u(\eps,\widehat{T}) + \liminf_{t\nearrow \widehat{T} }\frac{\eps}{\widehat{T}-t}=\infty. $$
By the minimum principle recalled in subsection \ref{prima}, we have $u_\eps \geq 0$ in $\widehat{{\mathcal{O}}}$. Letting $\eps$ go to zero we have $u_\eps\geq 0$ in $\widehat{{\mathcal{O}}}$, thus completing the proof. \end{proof}
\begin{proposition}\label{duedue} Let $\Omega\subseteq X$ be open and let $T_0$ and  $T\in \erre, $ such that \mbox{$0<T_0<T.$} Let ${\mathcal{O}}:=\Omega \times ]0,T[$ and 
$u: {\mathcal{O}} \ttende \erre$ be such that the restrictions $u|_{\Omega \times ]0,T_0[}$ and $u|_{\Omega \times ]T_0,T[}$ are $\cappa$-superharmonic. Then, if 
\begin{equation} \label{star} \liminf_{\substack{z\ttende (\xi,T_0)\\ (x,t)\in {\mathcal{O}}} } u (x,t)=  \liminf_{\substack{z\ttende (\xi,T_0)\\ t <T_0\\(x,t)\in {\mathcal{O}}} } u (x,t)= u(\xi,T_0)\quad \forall \xi\in\Omega,\end{equation} 
the function $u$ is $\cappa$-superharmonic in $\Omega\times]0,T[.$
\end{proposition} 
\begin{proof} Since $u$ is lower semicontinuous in $\Omega\times ]0,T_0[$ and in $\Omega\times ]T_0,T[$, the assumption \eqref{star} 
implies that u is lower semicontinuous in ${\mathcal{O}}=\Omega\times ]0,T[.$

To prove that $u$ is $\cappa$-harmonic in ${\mathcal{O}}$ we will show the following claim. \\
{\it Claim.\ }  For every $z\in {\mathcal{O}}$ there exists a basis $B_z$ of $\cappa$-regular neighborhoods of $V$ such that 
$$u(z)\geq K_\varphi^V(z)\qquad \forall\varphi\in C(\partial V,\erre), u|_{\partial V} \geq \varphi.$$
Here $K_\varphi^V$ denotes the unique $\cappa$-harmonic function in $V$, continuous up to $\partial V$ and such that 
$K_\varphi^V|_{\partial V}=\varphi.$ 

From this Claim our assertion follows thanks to Corollary 6.4.9 in \cite{BLU}. 

If $z\in \Omega \times ]0,T_0[$ or if $z\in \Omega \times ]0,T[$, the Claim is satisfied since $u$ is \mbox{$\cappa$-superharmonic} both in 
$\Omega \times ]0,T_0[$ and in $\Omega \times ]0,T[$. Then it remains to prove the Claim for every point $\zeta = (\xi, T_0), \xi \in \Omega.$ Let $B_\rho = (V)$ be a basis of $\cappa$-regular neighborhoods of $\zeta$ such that $\overline{V}\subseteq {\mathcal{O}}.$ Let $\varphi\in C(\partial V, \erre), \varphi \le u|_{\partial V}.$  Then $u - K_\varphi^V$ is $\cappa$-superharmonic in $\Omega \times ]0,T_0[$ and 
$$\liminf_{z\ttende z'} u(z) \geq u(z') - u(z') \geq 0\qquad \forall z'\in\partial_p \Omega \times ]0,T_0[.$$ 
Therefore, by Proposition \ref{pallino}, 
$$ u - K_\varphi^V\geq 0\inn V \inte \{t<T_0\}.$$ 
As a consequence, keeping in mind assumption \eqref{star},
$$ u(\xi,T_0) = \liminf_{\substack{(x,t) \ttende (\xi,\tau)\\ t<T_0 }}  u(x,t) \geq \liminf_{\substack{(x,t) \ttende (\xi,T_0)\\ t<T_0 }}  K_\varphi^V(x,t) =  K_\varphi^V(\xi,T_0),$$
that is,

$$ u(\xi,T_0) \geq K_\varphi^V(\xi,T_0).$$
This completes the proof.\end{proof}

\section{Some preliminary results} 

The proof of our main theorem rests on the following two lemmata.

\begin{lemma}\label{lemma31} Let $\Omega$ be a bounded open set such that $\overline{\Omega}\subseteq X$, and let \mbox{${\mathcal{O}}:=\Omega \times ]0,T[$}, $T\in\erre, T>0.$ Let $\varphi : \partial {\mathcal{O}} \ttende \erre$ be upper semicontinuous and such that $t\frecciaf \varphi (x,t)$ is monotone decreasing, $\forall x \in \partial \Omega$ and 
$$\varphi (x,0) = M = \sup_{\partial {\mathcal{O}}} \varphi\qquad (M\in\erre).$$ 
Then, the Perron solution $K_\varphi^{\mathcal{O}}$ is monotone decreasing w.r.t. the variable $t$: more precisely 
$$ t\frecciaf K_\varphi^{\mathcal{O}}(x,t) \mbox { is monotone decreasing for every fixed } x\in\Omega.$$ 
\end{lemma} 

\begin{proof} 
For every fixed $\delta\in ]0,T[$ let us define 
$$h(x,t)=K_\varphi^{\mathcal{O}}(x,t) - K_\varphi^{\mathcal{O}}(x,t+\delta), \ x\in\Omega, 0<t<T-\delta.$$
It is enough to prove that $h\geq 0$ in ${\mathcal{O}}_\delta:= \Omega\times ]0,T-\delta[.$ To this end we show that, for every $u\in \overline{\mathcal{U}}^{\mathcal{O}}_\varphi$ and $v\in \underline{\mathcal{U}}^{\mathcal{O}}_\varphi$, the function 
$$w(x,t)=u(x,t)-v(x,t+\delta)$$ is nonnegative in ${\mathcal{O}}_\delta$. Now, we have:
\begin{itemize}
\item[$(a)$] $w$ is $\cappa$-superharmonic in ${\mathcal{O}}_\delta$, since $u\in \overline{\mathcal{K}}({\mathcal{O}})$ and $(x,t)\frecciaf v(x,t+\delta)$ 
is $\cappa$-subharmonic in ${\mathcal{O}}_\delta$ being $v\in \underline{\mathcal{K}}({\mathcal{O}})$ and $\elle$  translation invariant in the variable $t$.

\item[$(b)$] For every $\overline{x} \in \Omega,$ 
\begin{eqnarray*} \liminf_{(x,t)\ttende (\overline{x},0)} w(x,t) &\geq&  \liminf_{(x,t)\ttende (\overline{x},0)} u(x,t) -  \liminf_{(x,t)\ttende (\overline{x},0)} v(x,t+\delta) \\ &\geq& \varphi(\overline{x}, 0) - v(\overline{x},\delta)\\&=& M-v(\overline{x},\delta)\geq 0.\end{eqnarray*} 

We remark that $v\le M$ in ${\mathcal{O}}$ since $v$ is $\cappa$-subharmonic and 

$$\limsup_{z\ttende \zeta} v(z) \le \varphi(\zeta) \le M \quad \forall\zeta\in \partial {\mathcal{O}}.$$
Here we use the maximum principle for subharmonic functions.

\item[$(c)$] 

For every $\zeta=(\xi,\tau)$, $\xi\in\partial\Omega, 0<\tau<T-\delta,$

$$ \liminf_{(x,t)\ttende (\xi,\tau)} w(x,t) \geq \varphi(\xi,\tau) - \varphi(\xi,\tau +\delta)\geq 0,$$ by hypotesis. 

\end{itemize}
From $(a)$, $(b)$ and $(c)$ and the minimum principle for superharmonic functions we get 
$$w\geq 0\inn {\mathcal{O}}_\delta.$$ This completes the proof. 
\end{proof}

With Lemma \ref{lemma31} at hand we can easily prove the following key result for our main theorem.

\begin{lemma}\label{lemma32} Let $\Omega$ be a bounded open set such that $\overline{\Omega} \subseteq X,$ and let ${\mathcal{O}}:=\Omega\times ]0,T[$, $T\in\erre$, $T>0.$  Let $z_0=(x_0,t_0)\in \partial\Omega\times ]0,T[$ be a $\elle$-regular boundary point.

 Then there exists a function $b\in\cappa({\mathcal{O}})$ such that 
\begin{itemize}
\item[$(i)$] $b$ is an $\elle$-barrier for ${\mathcal{O}}$ at $z_0$;

\item[$(ii)$] $t\frecciaf b(x,t)$ is monotone decreasing for every fixed $x\in\Omega.$ 

\end{itemize}

\end{lemma} 

\begin{proof} 
Let $Y$ be a bounded open set such that $\overline{\Omega} \subseteq Y\subseteq \overline{Y}\subseteq X$ and let $x_0\in\Omega$. 
By Lemma \ref{due!} there exists a function $h\in C^\infty(Y,\erre)$ such that 
\begin{itemize}
\item[$(a)$]  $h(x_0)=0$ and $h(x)>0\quad \forall x\neq x_0.$

\item[$(b)$]  $\elleo h >0$ in $\Omega$. 

\end{itemize}
For a fixed $\delta\in ]0,T_0[$ let us define 

\begin{equation*}  \widehat{h}:\overline{\Omega} \times  [0,T]  \ttende \erre,\quad \widehat{h}(x,t)= \begin{cases}
 h(x) \ifff \delta <t\le T,    \\  \ \, M \ \ifff 0 \le t\le \delta,  \end{cases}\end{equation*} 
where $M=\sup_{\overline \Omega} h.$ 

This function is $\elle$-superharmonic in ${\mathcal{O}}_1:= \Omega\times ]0,\delta[$ and in ${\mathcal{O}}_2:= \Omega\times ]\delta,T[$ since 
$$\elle \widehat{h}=0\inn {\mathcal{O}}_1\andd \elle \widehat{h}= \elleo h > 0\inn {\mathcal{O}}_2.$$ 
On the other hand,

$$ \limsup_{\substack{(x,t)\ttende (\xi,\delta)\\ t<\delta}}\widehat{h}(x,t)=M=  \limsup_{\substack{(x,t)\ttende (\xi,\delta)}}\widehat{h}(x,t).$$
Then, by Proposition \ref{duedue}, $$\widehat{h}\in \underline{\cappa} (\Omega\times ]0,T[).$$
Moreover, 
$$t\frecciaf \widehat{h}(x,t) \mbox{ \ is monotone decreasing,}$$ 
for every fixed $x\in\overline{\Omega}.$ 

Let us now put
$$b:=K^{\mathcal{O}}_{\widehat{h}|\partial {\mathcal{O}}},$$
which is well defined and $\cappa$-harmonic in ${\mathcal{O}}$, since $\widehat{h}|_{\partial {\mathcal{O}}}$ is bounded and upper semicontinuous.

Moreover, by Lemma \ref{lemma31}, $t\frecciaf b(x,t)$ is monotone decreasing for every fixed $x\in\Omega.$ 

It remains to show that $b$ is an $\elle$-barrier for ${\mathcal{O}}$ at $z_0$.  To this end we first remark that 
$$\widehat{h} \in \underline{\mathcal{U}}^{\mathcal{O}}_{\widehat{h}|_{\partial {\mathcal{O}}}},$$
so that 
$$\widehat{h}\le b\inn {\mathcal{O}}.$$
This implies $b>0$ in ${\mathcal{O}}$ since $\widehat{h}$ is strictly positive.

On the other hand, since $\widehat{h}|_{\partial {\mathcal{O}}}$ is continuous in a neighborhood of $z_0$, and $z_0$ is $\elle$-regular for ${\mathcal{O}}$, 
$$\lim_{z\ttende z_0} b(z)=\lim_{z\ttende z_0} K^{\mathcal{O}}_{\widehat{h}|_{\partial {\mathcal{O}}}} (z)= \widehat{h} (z_0)= \phi(x_0)=0.$$ 
This completes the proof.\end{proof}

\section{Proof of Theorem 1.1}
Let us keep the notation of Theorem 1.1  and split the proof in two steps.
\begin{itemize}
\item[(1)] {\it If $x_0\in\partial\Omega$ is $\elle_0$-regular for $\Omega$, then $z=(x_0,t_0)$ is $\elle$-regular for ${\mathcal{O}}.$}
\end{itemize}
Indeed, the $\elleo$-regularity of $x_0$ implies the existence of a $\elle_0$-harmonic barrier for $\Omega$ at $x_0$, i.e. a function $b_0\in\cappa(\Omega)$ such that $$b_0>0\inn \Omega \andd b_0\ttende 0 \ass x\ttende x_0.$$

It follows that 
$$\widehat{b}(x,t)=b_0(x),\quad (x,t)\in {\mathcal{O}},$$
is $\elle$-harmonic in ${\mathcal{O}}$ ($\elle \widehat{b} = \elleo b_0=0$). Moreover,

$$\widehat{b}>0\inn {\mathcal{O}}\andd \widehat{b}(x,t)=b_0(x)\ttende 0\ass (x,t)\ttende (x_0,t_0).$$

Hence, $\widehat{b}$ is an $\elle$-barrier function for ${\mathcal{O}}$ at $z_0$ and, as a consequence, $z_0$ is \mbox{$\elle$-regular} for ${\mathcal{O}}$.

\begin{itemize}
\item[(2)] {\it If  $z=(x_0,t_0)$, $x_0\in\Omega, 0<t_0<T,$  is $\elle$-regular for ${\mathcal{O}}$,  then $x_0$ is \mbox{$\elleo$-regular} for $\Omega$.}
\end{itemize}

Indeed, by Lemma \ref{lemma32}, there exists a function $b\in\cappa ({\mathcal{O}})$ such that $b>0,$\  \mbox{$b(z)\ttende 0$} as $z\ttende z_0$ and 

$$t\frecciaf b(x,t) \mbox{ \ is monotone decreasing}\quad \forall x\in \Omega.$$ 

It follows that, letting $b_0(x)=b(x,t_0),$ 
$$\elle_0 b_0= \elle b + \partial_t b = \partial_t b \le 0 \inn \Omega.$$ 
Hence,  $b_0$ is $\elleo$-superharmonic  in $\Omega.$ Moreover, $b_0>0$ in $\Omega$ and
$$b_0(x)=b(x,t_0)\ttende 0 \ass x\ttende x_0.$$
Therefore, $b_0$ is an $\elle$-barrier for $\Omega$ at $x_0$ , and $x_0$ is $\elle_0$-regular.

\section{An application to degenerate Ornstein--Uhlenbeck operators}

 In $\erre^N$ let us consider the partial differential operator 
 
 \begin{eqnarray}\label{OU}
L_0 = \mathrm{div}\left(  A\nabla\right)  +\left\langle B x,\nabla\right\rangle,
\end{eqnarray}
where  $A = (a_{ij})_{i,j = 1, \dots , N}$ and $B= (b_{ij})_{i,j = 1, \dots , N}$  are  $N\times N$  real constant  matrices, $x=(x_1, \ldots, x_N)$  is the point of $\erre^{N},$ $\mathrm{div}$, $\nabla$  and $\langle\ , \ \rangle$ denote the divergence, the Euclidean gradient  and the inner product in $\erre^{N}$, respectively.

We suppose that the matrix $A$ is symmetric, positive semidefinite and  that it assumes the following block form 

\begin{equation*}
A=%
\begin{bmatrix}
A_{0} & 0\\
0 & 0
\end{bmatrix},\label{A}%
\end{equation*}
$A_{0}$ being  a  $p_{0}\times p_{0}$ strictly positive definite matrix with $1\le p_{0}\leq N$. Moreover, we assume the matrix $B$ to be of the following type 

\begin{equation}
B=%
\begin{bmatrix}
0 & 0 & \ldots & 0  & 0 \\
B_{1}  & 0 &  \ldots & 0  & 0  \\
0 & B_{2}& \ldots &  0 &  0\\
\vdots & \vdots & \ddots & \vdots & \vdots\\
0 & 0 & \ldots &  B_{r} &  0
\end{bmatrix},
\label{B}%
\end{equation}
where $B_{j}$ is a $p_{j-1}\times p_{j}$ block with rank $p_{j}$ ($j=1,2,...,r$),  $p_{0}\geq p_{1}\geq...\geq p_{r}\geq1$ and $p_{0}+p_{1}+...+p_{r}=N$.

Finally, letting 

$$E(s):= \exp(-sB),\quad s\in\erre,$$

we assume that the following condition is satisfied

\begin{equation*}\label{C(t)}
C(t)=\int_{0}^{t} E(s)AE^T(s)\,ds \mbox{ \ is strictly positive definite for every $t>0$.} 
\end{equation*}

As it is quite well known this condition implies the hypoellipticity of $L$, see  \cite{lanconelli_polidoro_1994}. In that paper it is proved that the evolution counterpart of $L_0$, i.e. the operator 
$$L=L_0-\de_t \inn \erre^{N+1},$$ 
is left translation invariant and homogeneous of degree two on the homogeneous group 
$$\mathbb{K}= (\erre^{N+1}, \circ, \delta_\lambda)$$ 
with composition law $\circ$ defined as follows
$$(x,t)\circ (x',t')=(x'+E(t')x,t+t')$$
and dilation $\delta_\lambda,\lambda>0,$ of this kind

 \begin{eqnarray*}\delta_\lambda :\erre^{N+1}\ttende\erre^{N+1},\quad \delta_\lambda (x,t)&\,=&
\delta_\lambda (x^{(p_0)},x^{(p_1)},\ldots,x^{(p_r),t })\\&:=&(\lambda
x^{(p_0)},\lambda^3
x^{(p_1)},\ldots,\lambda^{2r+1}x^{(p_n)}, \lambda^2 t ),\end{eqnarray*}
where $x^{(p_i)}\in\erre^{p_i},\  i=0,\ldots,r.$

The natural number $q:=Q+2$, with 
\begin{equation}\label{Q} Q:=p_0+3p_1+\ldots + (2r+1) p_r, 
\end{equation}
is the homogenous dimension of $\mathbb{K}.$ In what follows we will write 

\begin{eqnarray*} \delta_\lambda (z)=\delta_\lambda (x,t)= (D_\lambda (x), \lambda^2 t ),\end{eqnarray*}
where, 

\begin{eqnarray*}D_\lambda (x)= (\lambda
x^{(p_0)},\lambda^3
x^{(p_1)},\ldots,\lambda^{2r+1}x^{(p_n)}, \lambda^2 t ).\end{eqnarray*}

Obviously,  $(D_\lambda)_{\lambda >0}$ is a group of dilations in $\erren$. The natural number $Q$ in \eqref{Q} is the homogeneous dimension of $\erren$ w.r.t. the group $(D_\lambda)_{\lambda >0}$. 

The operator $L$ has a fundamental solution $\Gamma$ given by 
$$\Gamma(z_0,z):= \gamma (z^{-1} \circ z_0),\quad z,\ z_0 \in \erre^{N+1},$$ where $\circ$ is the composition law in $\mathbb{K}$, 
$z^{-1}$ denotes the opposite of $z$ in $\mathbb{K}$ and, for a suitable $C_Q>0,$

\begin{equation*}  \gamma(x,t)= \begin{cases} 0\quad   \ifff  t\le 0,   \\  \\ \frac{C_Q}{t^Q} \exp \left(-\frac{1}{4} \left|D_{\frac{1}{\sqrt t}} (x)\right|_C^2\right)  \ \ifff t>0,   \end{cases}\end{equation*} 

where, 
$$|y|_C^2= \langle C^{-1}(1) y, y\rangle, $$
see again  \cite{lanconelli_polidoro_1994}.

It is quite easy to recognise that our Tikhonov-type theorem applies to the operators  $L_0$ and $L$.  Hence, if $\Omega$ is a bounded open subset of $\erre^N$, $x_0\in\partial\Omega$ and $t_0\in ]-T,T[,  T>0,$ we have:
\begin{eqnarray*} \mbox{\it $x_0$ is $L_0$-regular for $\Omega$}\end{eqnarray*} \begin{eqnarray*}  \mbox{ if and only if }   \end{eqnarray*}  \begin{eqnarray*}  \mbox{\it $z_0=(x_0,0)$ is $L$-regular for ${\mathcal{O}}_T:= \Omega\times ]-T,T[$}. \end{eqnarray*} 

On the other hand, in  \cite[Corollary 1.3]{kogoj_lanconelli_tralli_2017} it is proved that 
\begin{eqnarray*} \mbox{\it $z_0$ is $L$-regular for ${\mathcal{O}}_T$}\end{eqnarray*} 
if, for a $\mu\in ]0,1[,$ the following condition holds:

\begin{equation}\label{wl} \sum_{k=1}^\infty \frac{|{\mathcal{O}}^c_{T,k} (z_0)|}{\mu^{\alpha(k)\frac{Q+2}{Q}} }=\infty,\end{equation} 
 where $\alpha(k)=k\log k$, $|\cdot |$ denotes the Lebesque measure in $\erre^{N+1}$ and 
 
 $${\mathcal{O}}^c_{T,k} (z_0)= \left\{  z \neq {\mathcal{O}}_T\ : \ \left(\frac{1}{\mu}\right)^{\alpha(k)} \le \Gamma (z_0,z)\le \left(\frac{1}{\mu}\right)^{\alpha(k+1)} \right\}.$$
 We express now this condition in a more explicit form. To this end we let 
 
 \begin{equation}\label{serve?} A_k^c(x_0) =\left\{ (x,t)\in \erre^{N+1} \ |\ x\notin \Omega, \gamma(z^{-1}\circ (x,0)) \geq \left(\frac{1}{\mu}\right)^{\alpha(k)} \right\}.
 \end{equation} 
Then, 
\begin{eqnarray*} {\mathcal{O}}^c_{T,k}((x_0,0))&= &(A_k(x_0)\meno A_{k+1}(x_0)) \cup \left\{ \gamma = \left( \frac{1}{\mu}\right)^{\alpha (k+1)}\right\} \\ &\supseteq& A_k(x_0)\meno A_{k+1} (z_0). \end{eqnarray*}

Hence, denoting for the sake of brevity,
$$d_k = | A_k(z_0)| \andd \nu=\mu^{\frac{(Q+2)}{Q}},$$
condition \eqref{wl} is satisfied if 

\begin{equation} \label{strat}\sum_{k=1}^\infty \frac{d_k- d_{k+1}}{\nu^{\alpha(k)}} = \infty.
\end{equation} 
On the other hand, for every $p\in\enne,$ 
\begin{eqnarray*} \sum_{k=1}^\infty \frac{d_k- d_{k+1}}{\nu^{\alpha(k)}}\end{eqnarray*} \begin{eqnarray*} = \frac{d_1}{\nu^{\alpha(1)}} + d_2  \left( \frac{1}{\nu^{\alpha(2)}} -  \frac{2}{\nu^{\alpha(1)}}\right)+\cdots + d_p  \left( \frac{1}{\nu^{\alpha(p)}} -  \frac{2}{\nu^{\alpha(p-1)}}\right) - \frac{d_{p+1}}{\nu^{\alpha(p)}}
\end{eqnarray*} 
 \begin{eqnarray*}  \le (1 -\nu^{\log 2} ) \sum_{k=1}^p  \frac{d_k}{\nu^{\alpha(k)}}  - \frac{d_{p+1}}{\nu^{\alpha(p)}}.
\end{eqnarray*}
Then, since $ \dfrac{d_{p+1}}{\nu^{\alpha(p)}}\ttende 0$  as $p\tende \infty$ (as we will see later) condition \eqref{strat} is satisfied if 
\begin{equation} \label{strattt}\sum_{k=1}^\infty \frac{d_k}{\mu^{\alpha(k)}} = \infty.
\end{equation} 

Keeping in mind the very definition of $\Gamma$, we have that $ A_k(x_0)$ is equal to the following set 

\begin{equation*} \left\{ (x,t)\in \erre^{N+1} \ |\ x\in \Omega^c, t <0, \left|D_{\frac{1}{\sqrt{|t|}}} (x_0- E(|t|x)) \right|^2_C<
2Q\log \frac{(C_Q\mu^{\alpha(k)})^{\frac{2}{Q}}}{t}   \right\},
 \end{equation*} 
whereby, with the change of variables $y:=x_0-E(|t|)x,$ $\tau=-t$, we get 

\begin{eqnarray}\label{d} \quad d_k=\left| \left\{ (y,\tau)\ |\ \tau>0,\ y\in x_0 - E(\tau)(\Omega^c),\left|D_{\frac{1}{\sqrt{|\tau|}}}\right|_C^2   < 2Q\log \frac{R_k}{\tau}       \right \}\right|.\end{eqnarray} 
Here $R_k = (C_Q \mu^{\alpha(k)})^{\frac{2}{Q}}$ and $\Omega^c:=\erreu\meno \Omega$.

Therefore, 

\begin{eqnarray*} d_k &\le& \left| \left\{ (y,\tau)\ |\ \tau>0,\  \left|D_{\frac{1}{\sqrt{|\tau|}}}\right|_C^2   < 2Q\log \frac{R_k}{\tau}       \right \}\right|
\\ && \mbox{(using the change of variables $y=D_{\sqrt{R_k}} (\xi), \tau=R_k s)$} 
\\ &=& R_k^{\frac{Q+2}{Q}} \left| \left\{ (\xi,s)\ |\ s>0, \left| D_{\sqrt{\frac{1}{s}}} (\xi)\right| \le 2Q\log \frac{1}{s}\right\}\right|.
  \end{eqnarray*} 
Hence, for a suitable dimensional constant $C_Q^*>0,$ 

$$d_k \le C_Q^* \mu^{\alpha(k)\frac{Q+2}{Q}} = C_Q^* \nu^{\alpha(k)}.$$
Then,

$$0\le \frac{d_{p+1}}{\nu^{\alpha(p)} } \le C_Q^*  \mu^{\alpha(p+1)- \alpha(p)} \ttende 0 \ass p\ttende \infty,$$
since $0<\mu<1$  and  $\alpha(p+1)- \alpha(p)= p\log\frac{p+1}{p} + \log{(p+1)} \ttende \infty.$

We have completed the proof of the following criterion:\\
{ \it Let $L$ be the Ornstein--Uhlenbeck-type operator in \eqref{OU} and let $\Omega\subseteq \erren$ be a bounded open set. Then, a point $x_0\in\partial\Omega$ is $L$-regular for $\Omega$ if 
\begin{equation}\label{kavouri} \sum_{k=1}^\infty \frac{d_k(\Omega,x_0)}{\mu^{\alpha(k)\frac{Q+2}{2}}} =\infty,\end{equation}
where $d_k(\Omega,x_0):=d_k$ is defined in \eqref{d}.}

We note that condition \eqref{kavouri} holds  if $\Omega$ satisfies  the exterior cone-type condition introduced in \cite{kogoj_2019}.  Geometric boundary regularity criteria for wide classes of hypoelliptic evolution operators are also established in \cite{manfr}, \cite{LU},  \cite{ltu_2016} and \cite{kogoj_2017}.

\section*{Acknowledgment}

The author has been partially supported by the Gruppo Nazionale per l'Analisi Matematica, la Probabilit\`a e le
loro Applicazioni (GNAMPA) of the Istituto Nazionale di Alta Matematica (INdAM).

\bibliographystyle{alpha} 
%\bibliography{bibliografia}

\end{document}